\documentclass[pdflatex,sn-mathphys-num]{sn-jnl}

\usepackage{amssymb,amsmath,amsthm,float,color,graphicx, epsfig,subcaption}
\usepackage{multirow}
\usepackage{booktabs}

\def\matR{\mathbb{R}}
\def\O{\Omega}
\def\bu{\textit{\textbf{u}}}
\def\bv{\textit{\textbf{v}}}

\def\bw{\textit{\textbf{w}}}
\def\bf{\mathbf{f}}
\def\bg{\mathbf{g}}
\def\bn{\mathbf{n}}

\def\bF{\mathbf{F}}
\def\bP{\mathbf{P}}
\def\bL{\mathbf{L}}
\def\bH{\mathbf{H}}
\def\bW{\mathbf{W}}
\def\bV{\mathbf{V}}
\def\bX{\mathbf{X}}

\def\div{\textup{div} \ }

\def\matF{\mathcal{F}}

\def\matT{\mathcal{T}}
\def\matC{\mathcal{C}_{NL}}

\def\eps{\varepsilon}

\theoremstyle{thmstyleone}
\newtheorem{theorem}{Theorem}[section]
\newtheorem{corollary}{Corollary}[section]
\newtheorem{lemma}[theorem]{Lemma}
\newtheorem{prop}[theorem]{Proposition}
\theoremstyle{thmstylethree}
\theoremstyle{definition}
 
\theoremstyle{thmstyletwo}
\newtheorem{remark}[theorem]{Remark}

\newtheorem{assumption}{Assumption}

\begin{document}

\title[Title]{Finite element approximation of the stationary Navier-Stokes problem with non-smooth data}

\author[1]{\fnm{María Gabriela} \sur{Armentano}}

\author*[1]{\fnm{Mauricio} \sur{Mendiluce}}\email{mmendiluce@dm.uba.ar}

\affil[1]{\orgdiv{Departamento de Matemática}, \orgname{Facultad de Ciencias Exactas y Naturales, Universidad de Buenos Aires, IMAS - Conicet}, \orgaddress{\city{Buenos Aires}, \postcode{1428}, \country{Argentina}}}

\abstract{The aim of this work is to analyze the finite element approximation of the two-dimensional stationary Navier-Stokes equations with non-smooth Dirichlet boundary data. The discrete approximation is obtained by considering the Navier-Stokes system with a regularized boundary data. Based on the existence of the {\it very weak solution}, for the Navier–Stokes system with $L^2$ boundary data, and a suitable decomposition of this solution, we obtain a priori error estimates between the approximation of the Navier-Stokes system with  non-smooth data and the finite element solution of the associated regularized problem. These estimates allow us to conclude that our approach converges with optimal order.}

\keywords{Navier-Stokes equations, non-smooth data, finite elements}

\pacs[MSC Classification]{65N15, 65N30}

\maketitle

\section{Introduction}\label{sec1}

The numerical approximation of incompressible flows with non-smooth boundary data is a recurring challenge in computational fluid dynamics. In many physically relevant situations, such as internal flows with sharp corners or driven boundary conditions, the prescribed boundary velocity lacks the regularity typically assumed in theoretical and numerical analyses. A prototypical example is the lid-driven cavity flow, which remains a canonical benchmark for numerical methods for the Stokes and  Navier–Stokes equations.

In this work, we analyze the finite element approximation of the stationary Navier-Stokes problem with non-smooth Dirichlet data. When Dirichlet boundary data are sufficiently regular, the stationary Navier–Stokes equations admit weak solutions in the standard Sobolev spaces, and their finite element approximation is well understood (see, for example, \cite{wang2016,gunzburger1983,fix1983,zheng} and the references therein). However, if the boundary data belong only to $L^2$, the classical variational formulation breaks down. This lack of regularity prevents a direct application of standard finite element theory and raises fundamental questions regarding the convergence and accuracy of numerical approximations.

Our approach is based on a regularization of the boundary data combined with the notion of {\it very weak solutions} for the original problem. More precisely, the non-regular boundary data are approximated by a family of smooth functions, leading to a regularized Navier–Stokes problem that admits a classical weak solution. The finite element method is then applied to this regularized problem.

In recent years, several works have addressed elliptic and linear fluid problems with low-regularity boundary data. Finite element approximations for the Poisson problem under $L^2$ boundary conditions were studied in \cite{apel2016, apel2017}. The error estimation for the two-dimensional Stokes driven cavity problem was analyzed in \cite{CAI2009}. More recently, \cite{duran2020, hamouda2017} investigated this issue for the Stokes problem, and \cite{CAI2023} addressed it for linear elliptic problems.

Extending these ideas to the Navier–Stokes equations is non-trivial due to the presence of the nonlinear convective term, which prevents a straightforward generalization of the linear theory.

First, we introduce the notation that will be used throughout the paper. Let $s\geq0$ be a real number, $1\leq p\leq \infty$ and $D\subset\matR^2$ be an open bounded set, we denote $W^{s,p}(D)$ to the standard space of Sobolev in $D$, with $\|\cdot\|_{W^{s,p}(D)}$ and $|\cdot|_{W^{s,p}(D)}$ its norm and seminorm respectively (for more details, see, for instance, \cite{Adams1975}).
The space $W^{s,p}_0(D)$ is defined as the closure of $C^{\infty}$ functions with compact support in $D$, i.e, the clousure of $C^{\infty}_0 (\bar{D})$ in $W^{s,p}(D)$. For $p=2$ we write $H^s(D)=W^{s,2}(D)$ and $H^s_0(D)=W_0^{s,2}(D)$. $W^{-s,q}(D)$ is defined as the dual space of $W_0^{s,p}(D)$ with $\frac{1}{p}+\frac{1} {q}=1$. In $H^{-1}(D)$ we define the norm  $$\|v\|_{H^{-1}(D)}=\sup_{w\in H^1_0(D), w\neq 0}\frac{\int_{D}v\cdot w}{|w|_{H^1(D)}}.$$ 

On the other hand we will say that a function $f(x) = O(g(x))$ if there exists a constant $C>0$ such that $f(x)\leq Cg(x)$. As usual, we use bold characters to denote vector-valued functions and their corresponding functional spaces.

Let $\O\subset\matR^2$, be a Lipschitz domain with boundary $\Gamma=\partial\O$ and let $\bf\in\bH^{-1}(\O)$ and $\bg\in \bL^2(\Gamma)$ with 
    \begin{equation}\label{cond-compatibilidad}
        \int_{\Gamma}\bg\cdot\bn=0,
    \end{equation}
where $\bn$ denotes the outward unit vector normal to the boundary. The stationary Navier-Stokes problem is find $(\bu,p)$ such that:

\begin{equation}\label{ec-principales} \left\{
	 \begin{array}{rcl}
		-\nu\Delta\bu+(\bu\cdot\nabla)\bu+\nabla p  &= \bf & \textup{ in } \, \O \\
		\div\bu  &=0 & \textup{ in } \, \O \\
            \bu &=\bg & \textup{ on } \, \Gamma
	\end{array} \right.
\end{equation}

where $\bu = (u_1, u_2)$ is the velocity of the fluid, $p$ is the pressure, $\bf = (f_1, f_2)$ is the imposed body force, and $\nu > 0$ represents the viscosity.

\bigskip

If we assume  $\bg \in \bH^{1/2}(\Gamma)$, the typical approach to solving the problem using finite elements is to multiply the equations by test functions and integrate by parts, leading to a suitable variational form in the standard spaces $\bH^1(\O)$ for the velocity $\bu$ and $L_0^2(\O)$ for the pressure $p$ (see, for example, \cite{wang2016,fix1983,gunzburger1983}). The challenge arises when $\bg \notin \bH^{1/2}(\Gamma)$, as in this case the solution $\bu \notin \bH^1(\O)$, and thus the standard theory cannot be applied.

The approach to solving this problem is consider a regular approximation function of $\bg$, let us call $\bg_{\eps}$, and then compare the solution of the  problem \eqref{ec-principales}  with the numerical solution of the regular problem:

\begin{equation}\label{NV_dato_ge} \left\{
	 \begin{array}{rcc}
		-\nu\Delta\bv+(\bv\cdot\nabla)\bv+\nabla q  &= \bf & \textup{in } \, \O \\
		\div\bv  &=0 & \textup{in } \,\O \\
            \bv &=\bg_{\eps} & \textup{on } \,\Gamma
	\end{array} \right.
\end{equation}
Although a similar approach is given in \cite{CAI2009,duran2020,hamouda2017} for the Stokes problem,  the
presence of the nonlinear convective term in the Navier-Stokes equations introduces extra difficulties and requires a different treatment. To analyze the approximation, we first study the existence of solution to the problem \eqref{ec-principales} for a non-regular boundary data through the concept of a \textit{very weak solution}. To this end, we employ an useful decomposition of the solution given in \cite{marusic2000}.  This decomposition allows us to show that the difference between the \textit{very weak solution} of \eqref{ec-principales} and the solution of \eqref{NV_dato_ge} can be controlled by the approximation error between the non-regular data $\bg$ and the regularized data $\bg_{\eps}$. Then, we approximate the classical solution  of \eqref{NV_dato_ge} by finite elements and show the a priori error estimates under approriate assumptions of the finite element spaces and the meshes. 

The paper is organized as follows:  In section \ref{sec2: preliminares} we present some preliminary results, among which are the result of existence of a solution for problem \eqref{NV_dato_ge} and the existence of a \textit{very weak solution} of \eqref{ec-principales}. In section \ref{sec3: errores de aproximacion} we analyze the error between the \textit{very weak solution} of \eqref{ec-principales} and the weak solution of \eqref{NV_dato_ge}. In section \ref{sec4: finite elements} we study finite element approximations by solving problem \eqref{NV_dato_ge} and present a priori error estimates depending on the approximation of $\bg$ by a regularizing function. Finally, in section \ref{sec5: numerical experiments}
we analyze several numerical aspects
concerning the solution
of the discrete problem and present numerical experiments for the cavity flow problem that confirm the theorical prediction.

\section{Preliminaries}\label{sec2: preliminares}

In this section we introduce some notations and results that we will use throughout the paper. 

For our analysis, we first consider the following problem with regular data: 

\begin{equation}\label{prob_regular}
	 \left\{\begin{array}{rcc}
		-\nu\Delta\bv+(\bv\cdot\nabla)\bv+\nabla q  &= \bf & \textup{in } \,\O \\
		\div\bv  &=0 & \textup{in } \,\O \\
            \bv &=\bv_D & \textup{on } \, \Gamma
	\end{array}\right.
\end{equation}   
with $\bf\in\bH^{-1}(\O)$ and $\bv_D\in\bH^{1/2}(\Gamma)$.

A variational form of problem \eqref{prob_regular} consists in finding a pair $(\bv,q)\in \bH^1(\O)\times L_0^2(\O)$ solution of the equations, such that, $\bv|_\Gamma=\bv_D$ and

  \begin{equation}\label{ec_prueba_exist}
      \left\{\begin{array}{rll}
         a(\bv,\bw)+c(\bv,\bv,\bw)-b(\bw,q) &=\int_\O \bf\cdot\bw  & \forall\bw\in\bH^1_0(\O) \\
         b(\bv,t) &= 0  & \forall t\in L_0^{2}(\O) \\  
         \end{array}\right.
  \end{equation}

with
    $$a(\bv,\bw) = \nu\int_\O\nabla\bv:\nabla\bw, \quad b(\bw,q) = \int_\O q\div\bw, \, {\mbox{ and }}\, c(\bu,\bv,\bw) = \int_\O (\bu\cdot\nabla)\bv\cdot\bw .$$

Throughout the paper, we will denote by $C$ a generic positive constant, not necessarily the same at each occurrence. Now we present some useful  results \cite{girault1986, temam1977}.

\begin{lemma}\label{properties_of_c}
    The following properties hold for the operator $c$:
	\begin{enumerate}[(a)]
        \item $|c(\bu,\bv,\bw)|\leq \|\bu\|_{L^4(\O)}\|\nabla\bv\|_{L^2(\O)}\|\bw\|_{L^4(\O)}$ $\forall \bu,\bw\in \bL^4(\O)$ and $\bv\in\bH^1(\O)$.
	\item For $\bu,\bv,\bw\in\bH_0^1(\O)$ exists $\matC>0$ such that $$|c(\bu,\bv,\bw)|\leq \matC|\bu|_{H^1(\O)}|\bv|_{H^1(\O)}|\bw|_{H^1(\O)}.$$
	\item $c(\bu,\bv,\bv)=0$, $\forall \bu\in\bH^1(\O)$ with $\div\bu=0$ and $\forall\bv\in \bH^1_0(\O)$
		\item $c(\bu,\bv,\bw)=-c(\bu,\bw,\bv)$, $\forall \bu\in\bH^1(\O)$ with $\div\bu=0$ and $\forall\bv,\bw\in\bH^1_0(\O)$.
	\end{enumerate}
\end{lemma}

\begin{proof}
     Item (a) is shown in \cite[lemma 1.1, Chapter II]{temam1977} where it is also shown that, for $\bu,\bv,\bw\in \bH^1_0(\O)$, there exists $C>0$ such that $|c(\bu,\bv,\bw)|\leq C\|\bu\|_{H^1(\O)}\|\bv\|_{H^1(\O)}\|\bw\|_{H^1(\O)}$. Consequently, by applying Poincaré's inequality, we derive (b). Items (c) and (d) are proven in \cite[lemma 1.3, Chapter II]{temam1977}.
\end{proof}

\begin{lemma}\label{extension_g}
Given a function $\bv_D\in \bH^{1/2}(\Gamma)$ satisfying the condition \eqref{cond-compatibilidad} (i.e., $\int_{\Gamma} \bv_D \cdot \bn =0$), there exists, for any $\eps>0$, a function $\Psi=\Psi(\eps)\in \bH^1(\O)$ such that
    $$\div\Psi=0, \qquad \Psi|_{\Gamma}=\bv_D$$
    $$|c(\bw,\Psi,\bw)|\leq \eps |\bw|_{H^1(\O)}^2 \qquad \forall \bw\in \bH_0^1(\O), \, \div\bw=0.$$ 
\end{lemma}

\begin{proof}
    See \cite[Lemma 2.3, Chapter 4]{girault1986}.
\end{proof}

\begin{theorem}\label{teo-existyunic-NS}
    Let $\bf\in\bH^{-1}(\O)$ and $\bv_D\in\bH^{1/2}(\Gamma)$ satisfy \eqref{cond-compatibilidad}. Given $\eps>0$, if $$\frac{\matC\|\hat{\bf}\|_{H^{-1}(\O)}}{(\nu-\eps)^2}<1,$$
   with $\hat{\bf}=\bf+\nu\Delta\Psi-(\Psi\cdot\nabla)\Psi$, where $\Psi$ is given by the lemma \ref{extension_g}. Then, the problem \eqref{ec_prueba_exist} has unique solution $(\bv,q)\in \bH^1(\O)\times L^2_0(\O)$.   
\end{theorem}

\begin{proof}
  The proof follows the ideas given in  \cite[Theorem 2.3,  Chapter IV]{girault1986} and \cite[Theorems 1.5 and 1.6,  Chapter II]{temam1977}.  
  Let $\Psi\in \bH^1(\O)$ be the extension function of $\bv_D$, with $ \div \Psi=0$,  given by the previous lemma. We set $\bv=\hat{\bv}+\Psi$, with $\hat{\bv}\in\bH^1_0(\O)$. So,
  \begin{align*}
      a(\bv,\bw)+c(\bv,\bv,\bw)&=a(\hat{\bv},\bw)+a(\Psi,\bw)+c(\hat{\bv},\hat{\bv},\bw)\\
      &+c(\Psi,\hat{\bv},\bw)+c(\hat{\bv},\Psi,\bw)+c(\Psi,\Psi,\bw)
  \end{align*} 
 Since the bilinear form $b(\cdot,\cdot)$ satisfies the inf–sup condition (see \cite[Theorem 1.4, Chapter IV]{girault1986})  problem \eqref{ec_prueba_exist} is equivalent to find  $(\hat{\bv},q)\in \bW\times L^2_0(\O)$, with $\bW=\{\bw\in\bH^1_0: \div\bw =0\}$,  such that
  \begin{equation}\label{ec_prueba_exist2}
      \Tilde{a}(\hat{\bv},\hat{\bv},\bw)=\int_\O\hat{\bf}\cdot\bw \quad \forall \bw\in\bW,
  \end{equation}
  where $$\Tilde{a}(\bu,\bv,\bw)=a(\bv,\bw)+c(\bu,\bv,\bw)+c(\Psi,\bv,\bw)+c(\bv,\Psi,\bw).$$ and
  $$\hat{\bf}=\bf + \nu\Delta\Psi-(\Psi\cdot\nabla)\Psi$$
  
  We can observe that $\hat{\bf}\in\bH^{-1}(\O)$.
  
  The proof of existence follows as in Theorem IV.2.1 of \cite{girault1986}, noting that for all $\Tilde{\bw},\bw \in \bW$ $$\Tilde{a}(\Tilde{\bw},\bw,\bw)=\nu|\bw|_{H^1(\O)}^2+c(\bw,\Psi,\bw).$$
  
  where we use item (c) of lemma \ref{properties_of_c}, for all $\bw, \Tilde{\bw} \in\bW$. By Lemma \ref{extension_g}, if we take $\eps<\nu$, we have

  $$\Tilde{a}(\Tilde{\bw},\bw,\bw)\geq (\nu-\eps)|\bw|^2_{H^1(\O)}, \quad \forall \Tilde{\bw},\bw\in\bW.$$ 
     
  Hence, the form $\Tilde{a}(\cdot,\cdot,\cdot)$ is uniformily elliptic and so the problem \eqref{ec_prueba_exist2} has at least one solution $(\hat{\bv},p)\in \bW\times L^2_0(\O)$ \cite{girault1986}.

  Now we prove the uniqueness. If $\bv_1$ is a solution of \eqref{ec_prueba_exist}, then $\hat{\bv}_1=\bv_1-\Psi$ is a solution of \eqref{ec_prueba_exist2}. Taking $\bw=\hat{\bv}_1$, we have $$\nu|\hat{\bv}_1|^2_{H^1(\O)}=-c(\hat{\bv}_1,\Psi,\hat{\bv}_1)+\int_\O \hat{\bf}\cdot\hat{\bv}_1\leq \eps|\hat{\bv}_1|^2_{H^1(\O)}+\|\hat{\bf}\|_{H^{-1}(\O)}|\hat{\bv}_1|_{H^1(\O)},$$
  then we obtain the following a priori estimate
  \begin{equation}\label{estimation_nonhomog}
      |\hat{\bv}_1|_{H^1(\O)}\leq \frac{\|\hat{\bf}\|_{H^{-1}(\O)}}{\nu-\eps}
  \end{equation}
  Now, suppose that $\bv_1$ and $\bv_2$ are two solutions of \eqref{ec_prueba_exist}. Then, $\hat{\bv}_1=\bv_1-\Psi$ and $\hat{\bv}_2=\bv_2-\Psi$ satisfy \eqref{ec_prueba_exist2} and so for all $\bw\in\bW$
  \begin{align*}
      a(\hat{\bv}_1,\bw)+c(\hat{\bv}_1,\hat{\bv}_1,\bw)+c(\hat{\bv}_1,\Psi,\bw)+c(\Psi,\hat{\bv}_1,\bw) &= \int_\O\hat{\bf}\cdot\bw\\
      a(\hat{\bv}_2,\bw)+c(\hat{\bv}_2,\hat{\bv}_2,\bw)+c(\hat{\bv}_2,\Psi,\bw)+c(\Psi,\hat{\bv}_2,\bw) &= \int_\O\hat{\bf}\cdot\bw
  \end{align*}

  subtracting both equations, taking $\bw=\hat{\bv}_1-\hat{\bv}_2$ and using Lemma \ref{properties_of_c} item  d), we obtain
  $$\nu|\hat{\bv}_1-\hat{\bv}_2|^2_1+c(\hat{\bv}_1,\hat{\bv}_1,\hat{\bv}_1-\hat{\bv}_2)-c(\hat{\bv}_2,\hat{\bv}_2,\hat{\bv}_1-\hat{\bv}_2)+c(\hat{\bv}_1-\hat{\bv}_2,\Psi,\hat{\bv}_1-\hat{\bv}_2)=0.$$

   Now, adding and subtracting $c(\hat{\bv}_1,\hat{\bv}_2,\hat{\bv}_1-\hat{\bv}_2)$ and using again Lemma \ref{properties_of_c} we get

   $$\nu|\hat{\bv}_1-\hat{\bv}_2|^2_{H^1(\O)}=-c(\hat{\bv}_1-\hat{\bv}_2,\hat{\bv}_2,\hat{\bv}_1-\hat{\bv}_2)-c(\hat{\bv}_1-\hat{\bv}_2,\Psi,\hat{\bv}_1-\hat{\bv}_2),$$ 

   $$\nu|\hat{\bv}_1-\hat{\bv}_2|^2_{H^1(\O)}\leq \matC|\hat{\bv}_1-\hat{\bv}_2|^2_{H^1(\O)}|\hat{\bv}_2|_{H^1(\O)}+\eps|\hat{\bv}_1-\hat{\bv}_2|^2_{H^1(\O)}.$$

Finally,  using \eqref{estimation_nonhomog} we obtain,

    $$\left(\nu-\frac{\matC\|\hat{\bf}\|_{H^{-1}(\O)}}{\nu-\eps}-\eps\right)|\hat{\bv}_1-\hat{\bv}_2|_{H^1(\O)}\leq 0,$$    
and since $\frac{\matC\|\hat{\bf}\|_{H^{-1}(\O)}}{(\nu-\eps)^2}<1$ we conclude that $\hat{\bv}_1-\hat{\bv}_2=0.$
\end{proof}

\subsection{Existence of \textit{very weak solution}}

In this subsection we analyze the existence of a solution to the problem \eqref{ec-principales} through the concept of \textit{very weak solution}. 
For $\bg\in \bL^2(\Gamma)$ and $\bf\in\bH^{-1}(\O)$, we define $$\bX=\{\phi\in \bW^{2,4/3}(\O)\cap \bW^{1,4/3}_0(\O); \ \div \phi=0\}$$ $$M=\{q\in W^{1,4/3}(\O); \quad \int_{\O}q=0\}.$$ We will say that $\bu\in \bL^4(\O)$ is \textit{very weak solution} of (\ref{ec-principales}) if
     \begin{align*}
     -\nu\int_{\O}\bu\cdot\Delta\phi-\int_{\O}(\bu\cdot\nabla)\phi\cdot\bu&=\int_\O\bf\cdot\phi - \int_{\Gamma}\bg\cdot\frac{\partial \phi}{\partial \bn} \quad \forall \phi\in \bX\\
     \int_{\O}\bu\cdot\nabla q &= \int_{\Gamma} (\bg\cdot\bn) q \quad \forall q\in M
     \end{align*}

\begin{remark}
    For linear problems, like the Poisson problem \cite{apel2016} or the Stokes problem \cite{hamouda2017, duran2020} with Dirichlet condition in $L^2(\Gamma)$, the  corresponding \textit{very weak solution} $\bu$ belongs to the space $\bL^2(\O)$. However, in our problem if $\bu\in \bL^2(\O)$ to ensure that
  
    $$\int_{\O}(\bu\cdot\nabla)\bv\cdot\bu=\sum_{i,j}\int_{\O}u_i\left (\frac{\partial v_j }{\partial x_i}\right)u_j$$ 
    is well-defned we need 
 $\bv\in \bW^ {1,\infty}(\O)$ and so, the $L^2$ setting is unsuitable for formulating our nonlinear problem.
\end{remark}

\begin{remark}\label{veryweak-->weak}
We note that, if we consider the solution of the variational problem with a regular Dirichlet data (for example, in $\bH^{1/2}(\Gamma)$), in which case $(\bu,p)\in \bH^1(\O)\times L^2_0(\O)$, then $\bu$ is also a \textit{very weak solution}. Indeed, from 1.4.4.5 in \cite{Grisvard1985}, we have that $W^{2,4/3}(\O)\subset W^{1,4}(\O)$, and since $L^4(\O)\subset L^2(\O)$, we can conclude that if $\phi \in \bX$ then $\phi\in \bH^1(\O)$. On the other hand, since $W^{1,4/3}(\O)\subset L^4(\O)$ (see 1.4.4.5 in \cite{Grisvard1985} again), we can also deduce that if $q\in M$, then $q\in L^2_0(\O)$.

\end{remark}

For our subsequent analysis, it is important to note that $\bH^1(\O)\hookrightarrow \bL^4(\O)$ \cite{Adams1975, Grisvard1985}. Furthermore, given $\bv\in\bH^1(\O)$ there exists $C_I>0$ such that 
\begin{equation}\label{inmersion}
    \|\bv\|_{L^4(\O)}\leq C_I \|\bv\|_{H^1(\O)}  
\end{equation}

Assuming that there exists a \textit{very weak solution}, the existence of $p$ can be proven using the following lemma.

\begin{lemma}\label{existence-p}
  Let $m$ a nonnegative integer and $s$ a real number such that $1<s<\infty$. Let $\matF\in \bW^{-m,s}(\O)$ such that: $$\langle \matF,\Phi \rangle = 0, \quad \forall \Phi\in \Lambda =\{ \xi\in (C^\infty_c(\O))^2: \div\xi =0\}$$ where $\langle\matF,\Phi\rangle=\int_\O\matF\cdot\Phi$. Then there exists $p\in W^{-m+1,s}(\O)$ such that $\matF = \nabla p$. 
\end{lemma}

\begin{proof}
    See \cite[Theorem 2.8]{amrouche-girault}.
\end{proof}

Now, taking $$\matF=\bf+\Delta\bu-(\bu\cdot\nabla)\bu,$$
we observe that $\matF \in \bW^{-2,4}(\O)$
and $\langle \matF,\phi \rangle=0 \quad \forall \phi\in \bX$. Therefore, from the previous lemma, we can conclude that there exists $p\in W^{-1,4}(\O)$ such that $\matF=\nabla p$.

Within our analytical framework, certain regularity assumptions on the solution of the classical Stokes problem are required.

\begin{assumption} \label{as:as1}
Assume that the Lipschitz domain $\O$ is such that:
given $\bF\in \bL^{4/3}(\Omega)$ the Stokes problem
    \begin{equation*}       
  \label{Stokes}
	 \left\{\begin{array}{rcc}
		-\nu\Delta\bw+\nabla \pi  &= \bF & \textup{in } \, \O \\
		\div\bw  &=0 & \textup{in } \,  \O \\
            \bw &= 0 & \textup{on } \, \Gamma
	\end{array}\right.
\end{equation*}
has a unique solution $(\bw,\pi)\in \bW^{2,4/3}(\O) \times W^{1,4/3}(\O)\setminus\matR$ and there exists a constant $C_S$ such that:
$$\| \bw \|_{W^{2,4/3}(\O)} + \| \pi \|_ {W^{1,4/3}(\O) }\leq C_S \| F \|_{L^{4/3}(\Omega)}$$

\end{assumption}

\begin{remark}
As stated in \cite{marusic2000}, if $\O$ is a bounded $C^{1,1}$ domain, the Assumption  \ref{as:as1} is satisfied according to Theorem 3 of \cite{amrouche-girault}. The assumption also holds if $\O$ is a convex polygon  (see  \cite[Theorem 5.4, Reamark 5.6]{girault1986} and \cite{grisvard1979}).
\end{remark}

From now on, we assume that $\Omega$ satisfies \textbf{Assumption 1}.
\medskip

The first important result to consider from Maru\v{s}i\'c-Paloka work \cite{marusic2000} is the existence and uniqueness of the \textit{very weak solution} $\bu$ for the case $\bf=0$.

\begin{theorem}\label{teo-exist-1}
Let $\bg\in \bL^2(\Gamma)$ satisfying \eqref{cond-compatibilidad}, $\bf=0$ and $\bL_\Gamma = \{\bw\in\bL^4(\O), \quad \div\bw=0, \quad \bw\cdot\bn\in L^2(\Gamma)\}$. If $\|\bg\|_{L^2(\Gamma)}$ is small enough, then there exists a unique \textit{very weak solution} $\bu\in\bL_\Gamma$ of \eqref{ec-principales}. Furthermore, there exists a constant $C_1=C_1(\O)>0$ such that $$\|\bu\|_{L^4(\O)}\leq \frac{C_1\|\bg\|_{ L^2(\Gamma)}}{1-C_1(1/\nu)\|\bg\|_{L^2(\Gamma)}}$$
\end{theorem}

\begin{proof}

The result follows from \cite[Theorem 4]{marusic2000}. Indeed, if the Assumption \ref{as:as1} is satisfied, the estimate provided in \cite[Lemma 2]{marusic2000} can be obtained. Consequently, as the author note in their Remark 3, all the results from that work are applicable.
\end{proof}

\section{Error analysis of approximating boundary data with regular functions}\label{sec3: errores de aproximacion}

In this section, we obtain estimates of the error involved in approximating our problem \eqref{ec-principales} with one that uses regular boundary data. To achieve this, we
show some useful decomposition of the solution of \eqref{ec-principales}, into a regular part and a non-regular part.

Let $\bg_\eps \in \bH^{1/2}(\Gamma)$ be a regularization of $\bg$ such that $$\|\bg-\bg_\eps\|_{L^2(\Gamma) }\rightarrow 0 \qquad \eps\rightarrow 0.$$ There are some proposals for $\bg_\eps$ that can be found in \cite{apel2016,duran2020,hamouda2017} which we will analyze in the numerical experiments section.

In \cite{marusic2000} the author present the following decomposition of the \textit{very weak solution} of \eqref{ec-principales}, which will be very useful for our objectives.

\begin{theorem}\label{marusic2}
     Given $\eps>0$, $\bf\in\bH^{-1}(\O)$ and $\bg_\eps \in \bH^{1/2}(\Gamma)$ a regularization of $\bg$, there exists $\bu$ \textit{very weak solution} of \eqref{ec-principales} which can be decomposed as $\bu=\bu_\eps+\bv_\eps$, where $\bu_\eps\in\bL_ {\Gamma}$ is a \textit{very week solution} of 
 \begin{equation}\label{non_regular_prob}
  \left\{\begin{array}{rcl}
  -\nu\Delta\bu_\eps+(\bu_\eps\cdot\nabla)\bu_\eps+\nabla p_\eps =&0 & \quad \textup{in } \,  \O \\
  \div\bu_\eps =&0 & \quad \textup{in } \, \O \\
 \bu_\eps =& \bg-\bg_\eps & \quad \textup{on } \,  \Gamma
  \end{array}\right.
 \end{equation}
 and $\bv_\eps\in\bH^1(\O)$ is a weak solution of
 \begin{equation}\label{prob-semiregular}
  \left\{\begin{array}{rcl}
  -\nu\Delta\bv_\eps+(\bv_\eps\cdot\nabla)\bv_\eps+(\bu_\eps \cdot\nabla)\bv_\eps+(\bv_\eps\cdot\nabla)\bu_\eps+\nabla q_\eps &= \bf & \quad \textup{in } \, \O \\
  \div\bv_\eps &=0 & \quad \textup{in } \,  \O \\
 \bv_\eps &=\bg_\eps & \quad \textup{on } \, \Gamma
  \end{array}\right.
 \end{equation}
\end{theorem}
\begin{proof}
Since $\bg-\bg_\eps\in \bL^2(\Gamma)$ satisfies the compatibility condition and we can choose $\eps>0$ as small as we like, then the existence and uniqueness of \eqref{non_regular_prob} are guaranteed by Theorem \ref{teo-exist-1}. We also have the following estimate: 

\begin{equation}\label{estimate_u_eps}
    \|\bu_\eps\|_{L^4(\O)}\leq  \frac{C_1\|\bg-\bg_\eps\|_{L^2(\Gamma)}}{1- C_1(1/\nu)\|\bg-\bg_\eps\|_{L^2(\Gamma)}}    
\end{equation}
    
 For the existence of a solution to the problem \eqref{prob-semiregular} we first observe that  given $\phi\in \bH^1_0(\O)$
 \begin{align}\label{integrals_u_eps}
 \int_\O (\bv_\eps\cdot\nabla)\bu_\eps\cdot\phi := -\int_\O (\bv_\eps\cdot\nabla)\phi\cdot\bu_\eps &\leq \ \|\bv_\eps\|_{L^4(\O)}\|\nabla\phi\|_{L^2(\O)}
 \|\bu_\eps\|_{L^4(\O)} \nonumber\\
 &\leq C_{I}\|\bv_\eps\|_{H^1(\O)}\|\phi\|_{H^1(\O)}\|\bu_\eps\|_{L^4(\O)}
 \end{align}

Besides,
 $$\int_\O(\bu_\eps\cdot\nabla)\bv_\eps\cdot\phi\leq C_I \|\bu_\eps\|_{L^4(\O)}\|\bv_\eps\|_{H^1(\O)}\|\phi\|_{H^1(\O)}$$
 Therefore, the linear terms $\int_\O (\bv_\eps\cdot\nabla)\bu_\eps\cdot\phi$ and $\int_\O(\bu_\eps\cdot\nabla)\bv_\eps\cdot\phi$ are well defined.

Now, the same ideas applied to the proof of existence of a solution to the problem \eqref{prob_regular} are used. Indeed, given $\phi\in \bH^1_0(\O)$ with $\div\phi=0$, it is considered
    $$a(\bv_\eps,\phi)+c(\bv_\eps,\bv_\eps,\phi)+\int_\O (\bu_\eps\cdot\nabla)\bv_\eps\cdot\phi+\int_\O (\bv_\eps\cdot\nabla)\bu_\eps\cdot\phi=\int_\O \bf\cdot\phi.$$
    Using the lemma \ref{extension_g}, let $\hat{\bv}_\eps=\bv_\eps-\Psi$ and considering
    \begin{align*}
        &a(\hat{\bv}_\eps,\phi)+c(\hat{\bv}_\eps,\hat{\bv}_\eps,\phi)+c(\hat{\bv}_\eps,\Psi,\phi)\\
        &+c(\Psi,\hat{\bv}_\eps,\phi)+\int_\O(\bu_\eps\cdot\nabla)\hat{\bv}_\eps\phi+\int_\O(\hat{\bv}_\eps\cdot\nabla)\bu_\eps\cdot\phi\\
        &=\int_\O \hat{\bf}\cdot\phi-\int_\O(\bu_\eps\cdot\nabla)\Psi\cdot\phi-\int_\O(\Psi\cdot\nabla)\bu_\eps\cdot\phi.
    \end{align*}
    where $\hat{\bf}=\bf+\nu\Delta\Psi-(\Psi\cdot\nabla)\Psi$ and the right-hand side of the equality is well-defined due to observation \eqref{integrals_u_eps}. 
    Then, we can apply arguments analogous to those provided in the proof of Theorem  \ref{teo-existyunic-NS}. 
    
    Finally, let us see that $\bu=\bu_\eps+\bv_\eps$ is a \textit{very weak solution} of the problem \eqref{ec-principales}. It is clear that $\bu \in \bL^4(\O)$, so we need to check that:
    \begin{align*}
        -\nu\int_{\O}(\bu_\eps+\bv_\eps)\cdot\Delta\phi -\int_{\O}((\bu_\eps+\bv_\eps)\cdot\nabla)\phi\cdot(\bu_\eps+\bv_\eps) &=\int_\O\bf\cdot\phi \\
        &-\int_{\Gamma}\bg\cdot\frac{\partial \phi}{\partial \bn} \quad \forall \phi\in \bX\\
        \int_{\O}(\bu_\eps+\bv_\eps)\cdot\nabla q &=\int_{\Gamma} (\bg\cdot\bn) q \quad \forall q\in M.
    \end{align*}  
        
    From the first equation, if we subtract and add $\bg_\eps$ in the integral that corresponds to the edge, and using observation \eqref{integrals_u_eps}, we can write
    \begin{align*}
       &\underbrace{-\nu\int_{\O} \bu_\eps\cdot\Delta\phi -\int_{\O}(\bu_\eps\cdot\nabla)\phi\cdot\bu_\eps +\int_{\Gamma}(\bg-\bg_\eps)\cdot\frac{\partial \phi}{\partial \bn}}_{(I)}\\
        &\underbrace{-\nu\int_{\O}\bv_\eps\cdot\Delta\phi +\int_{\O}(\bv_\eps\cdot\nabla)\bv_\eps\cdot\phi +\int_{\O}(\bu_\eps\cdot\nabla)\bv_\eps\cdot\phi +}_{(II)}\\
        &\underbrace{+\int_{\O}(\bv_\eps\cdot\nabla)\bu_\eps\cdot\phi+\int_{\Gamma}\bg_\eps\cdot\frac{\partial \phi}{\partial \bn}}_{(II)}=\int_\O\bf\cdot\phi 
    \end{align*}
    
    Given that $\bu_\eps$ is a \textit{very weak solution} of the Navier-Stokes problem with edge data $\bg-\bg_\eps$, term (I) is equal to zero. On the other hand, term (II) must be equal to $\int_\O \bf\cdot\phi$ since $\bv_\eps$ is a \textit{very weak solution} of problem \eqref{prob-semiregular} (see Remark \ref{veryweak-->weak}). For the second equation the reasoning is analogous.
\end{proof}

Next, we present the first main result, which shows the approximation between $\bv_\eps$ (the regular part of the decomposition of the \textit{very weak solution} $\bu$ of \eqref{ec-principales}) and the solution $\bv$ of \eqref{NV_dato_ge} (i.e, the solution of the classical Navier-Stokes problem with boundary data $\bg_\eps$).

\begin{prop}\label{main_result_1}
     Given $\eps>0$ and let $(\bv,q)\in \bH^1(\O)\times L^2_0(\O)$ be the weak solutions of \eqref{NV_dato_ge}.
    If $$\frac{\matC\|\hat{\bf}\|_{H^{-1}(\O)}}{(\nu-\eps)^2}<1,$$
    
    with $\hat{\bf}=\bf+\nu\Delta\Psi-(\Psi\cdot\nabla)\Psi$, where $\Psi \in \bH^1(\O)$ is the regular extension function of $\bg_{\eps}$ provided by lemma \ref{extension_g}. Then, there exists a positive constant $C$ such that $$|\bv_\eps-\bv|_{H^1(\O)}\leq C\frac{C_1\|\bg-\bg_\eps\|_{L^2(\Gamma)}}{1-C_1(1/\nu)\|\bg-\bg_\eps\|_{L^2(\Gamma)}}.$$ 
\end{prop}

\begin{proof}
    For $\bg_{\eps} \in \bH^{1/2} (\Gamma)$, by Lemma \ref{extension_g} we know that there exists a function $\Psi\in\bH^1(\O)$ such that $\div \Psi=0$ and $\Psi|_{\Gamma}=\bg_\eps$. Following the ideas of the proof of the Theorem \ref{teo-existyunic-NS}, we can write $\bv_\eps=\hat{\bv}_\eps+\Psi$ and $\bv=\hat{\bv}+\Psi$ with $\hat{\bv}_\eps,\hat{\bv}\in \bH^1_0(\O)$. Given $\phi\in\bH^1_0(\O)$, with $\div\phi=0$, a test function for the variational problem of \eqref{NV_dato_ge} and \eqref{prob-semiregular}, 

    \begin{equation}\label{weak_form_v_eps-hat}
        a(\hat{\bv}_\eps,\phi)+c(\hat{\bv}_\eps,\hat{\bv}_\eps,\phi)+c(\hat{\bv}_\eps,\Psi,\phi)+c(\Psi,\hat{\bv}_\eps,\phi)=\int_\O \hat{\bf}\cdot\phi -\int_\O(\bu_\eps\cdot\nabla)\bv_\eps\cdot\phi-\int_\O(\bv_\eps\cdot\nabla)\bu_\eps\cdot\phi    
    \end{equation}

    $$a(\hat{\bv},\phi)+c(\hat{\bv},\hat{\bv},\phi)+c(\hat{\bv},\Psi,\phi)+c(\Psi,\hat{\bv},\phi)=\int_\O \hat{\bf}\cdot\phi,$$
    where in particular we use that $(\bv,q)$ satisfies
    \begin{equation}\label{forma_debil_regular}
        a(\bv,\bw)+c(\bv,\bv,\bw) =\int_\O \bf\cdot\bw  \quad \forall\bw\in\bV
    \end{equation}

Subtracting both equations and taking $\phi=\hat{\bv}_\eps-\hat{\bv}$ we can obtain
\begin{align*}
  a(\hat{\bv}_\eps-\hat{\bv},\hat{\bv}_\eps-\hat{\bv})&+c(\hat{\bv}_\eps-\hat{\bv},\hat{\bv},\hat{\bv}_\eps-\hat{\bv})+c(\hat{\bv}_\eps-\hat{\bv},\Psi,\hat{\bv}_\eps-\hat{\bv})=\\
  &-\int_\O(\bu_\eps\cdot\nabla)\bv_\eps\cdot(\hat{\bv}_\eps-\hat{\bv})-\int_\O(\bv_\eps\cdot\nabla)\bu_\eps\cdot(\hat{\bv}_\eps-\hat{\bv})  
\end{align*}

Then,
\begin{align*}
\nu|\hat{\bv}_\eps-\hat{\bv}|^2_{H^1(\O)}&=-c(\hat{\bv}_\eps-\hat{\bv},\hat{\bv},\hat{\bv}_\eps-\hat{\bv})-c(\hat{\bv}_\eps-\hat{\bv},\Psi,\hat{\bv}_\eps-\hat{\bv})\\
&-\int_\O(\bu_\eps\cdot\nabla)\bv_\eps\cdot(\hat{\bv}_\eps-\hat{\bv})   -\int_\O(\bv_\eps\cdot\nabla)\bu_\eps\cdot(\hat{\bv}_\eps-\hat{\bv})
\end{align*}

Using the properties of the operator $c(\cdot,\cdot,\cdot)$, the estimate given in lemma \ref{extension_g}, the estimate \eqref{integrals_u_eps}, the a priori estimate \eqref{estimation_nonhomog} for $\hat{\bv}$ and choosing $\eps$ such that $\nu-\eps>0$ we get

\begin{align*}
    \nu|\hat{\bv}_\eps-\hat{\bv}|^2_{H^1(\O)} &\leq \matC|\hat{\bv}_\eps-\hat{\bv}|^2_{H^1(\O)}|\hat{\bv}|_{H^1(\O)}+\eps|\hat{\bv}_\eps-\hat{\bv}|^2_{H^1(\O)}\\
    &+C\|\bu_\eps\|_{L^4(\O)}\|\bv_\eps\|_{H^1(\O)}\|\hat{\bv}_\eps-\hat{\bv}\|_{H^1(\O)}\\
    &\leq|\hat{\bv}_\eps-\hat{\bv}|^2_{H^1(\O)}\frac{\matC\|\hat{\bf}\|_{-1}}{\nu-\eps}+\eps|\hat{\bv}_\eps-\hat{\bv}|^2_{H^1(\O)}\\
    &+C\|\bu_\eps\|_{L^4(\O)}\|\bv_\eps\|_{H^1(\O)}|\hat{\bv}_\eps-\hat{\bv}|_{H^1(\O)}
\end{align*}

Now, by taking $\phi=\hat{\bv}_\eps$ into \eqref{weak_form_v_eps-hat}, applying  property (c) of lemma \ref{properties_of_c}, lemma \ref{extension_g}, the estimation \eqref{integrals_u_eps}, 
 and the fact that $\|\bu_\eps\|_{L^4(\O)}$ is bounded,  we can get 
 $$|\hat{\bv}_\eps|_{H^1(\O)}\leq C( \|\hat{\bf}\|_{H^{-1}(\O)} + \|\Psi\|_{H^1(\O)}).$$
and so, $$\|\bv_\eps\|_{H^1(\O)}=\|\hat{\bv}_\eps+\Psi\|_{H^1(\O)}\leq C(|\hat{\bv}_\eps|_{H^1(\O)}+\|\Psi\|_{H^1(\O)})\leq C(\|\hat{\bf}\|_{H^{-1}(\O)} + \|\Psi\|_{H^1(\O)}).$$

Finally, we can infer
$$\left(\nu-\frac{\matC\|\hat{\bf}\|_{-1}}{\nu-\eps}-\eps\right)|\hat{\bv}_\eps-\hat{\bv}|_{H^1(\O)}\leq C\|\bu_\eps\|_{L^4(\O)}$$
where $\nu-\frac{\matC\|\hat{\bf}\|_{-1}}{\nu-\eps}-\eps>0$ since $\frac{\matC\|\hat{\bf}\|_{-1}}{(\nu-\eps)^2}<1$.

Thus, using \eqref{estimate_u_eps} in the above estimation, we conclude the proof.

\end{proof}

We note that, since $\bv_\eps-\bv\in \bH^1_0(\O)$, we can rewrite the estimate proven in the previous theorem using the Poincaré inequality and obtain

\begin{equation}\label{estimate_v-v_eps}
\|\bv_\eps-\bv\|_{H^1(\O)}\leq C\frac{C_1\|\bg-\bg_\eps\|_{L^2(\Gamma)}}{1-C_1(1/\nu)\|\bg-\bg_\eps\|_{L^2(\Gamma)}},    
\end{equation}

for some positive constant $C$.

Now, we are in a position to present the following main result.

\begin{theorem}\label{main_result_2}
    Given $\eps>0$ and $\bg_\eps\in \bH^{1/2}(\Gamma)$ such that $\|\bg -\bg_\eps\|_{L^2(\Gamma)}<\eps$. Assuming that $\frac{\matC\|\hat{\bf}\|_{-1}}{(\nu-\eps)^2}<1$ and $\|\bg\|_{L^2(\Gamma)}$ is small enough, if $\bu\in \bL_{\Gamma}$ is \textit{very weak solution} of \eqref{ec-principales} and $\bv$ is a weak solution of \eqref{NV_dato_ge} then $$\|\bu- \bv\|_{L^4(\O)}\leq C\frac{C_1\|\bg-\bg_\eps\|_{L^2(\Gamma)}}{1-C_1(1/ \nu)\|\bg-\bg_\eps\|_{L^2(\Gamma)}}$$ with $C$ a positive constant.
\end{theorem}
\begin{proof}
    Since $\bu$ is a \textit{very weak solution}, by the decomposition given in theorem \ref{marusic2}  we can write $\bu=\bu_\eps + \bv_\eps$. Hence, by \eqref{inmersion} and the estimates \eqref{estimate_u_eps} and \eqref{estimate_v-v_eps} we get
    
    \begin{align*}
        \|\bu-\bv\|_{L^4(\O)}&=\|\bu_\eps+\bv_\eps-\bv\|_{L^4(\O)}\leq\|\bu_\eps\|_{L^4(\O)}+\|\bv_\eps-\bv\|_{L^4(\O)}\\
        &\leq\|\bu_\eps\|_{L^4(\O)}+C_I\|\bv_\eps-\bv\|_{H^1(\O)}\\
        &\leq \frac{C_1\|\bg-\bg_\eps\|_{L^2(\Gamma)}}{1-C_1(1/\nu)\|\bg-\bg_\eps\|_{L^2(\Gamma)}} + C_IC\frac{C_1\|\bg-\bg_\eps\|_{L^2(\Gamma)}}{1-C_1(1/\nu)\|\bg-\bg_\eps\|_{L^2(\Gamma)}}
    \end{align*}
\end{proof}

\section{Finite element approximation and a priori estimates}\label{sec4: finite elements}

In this section, we analyze the numerical approximation of the solution of problem \eqref{ec-principales} using finite element approximations of problem \eqref{NV_dato_ge}, with appropriate regular data $\bg_\eps$.

Let $\O$ be a domain satisfying Assumption \ref{as:as1} and with a $C^{k+1}$ piecewise smooth boundary, with $k\geq 1$ sufficiently large to fulfill our requirements. Let $\matT_h$ be a partition of  $\O$ which represent convincingly the domain, i.e., we use curved triangles along the curved part of the boundary (as, for example, the curved triangles introduced by Zlamal in \cite{zlamal1973}) and classical triangles in the rest of the domain. We use the same notation to call the curved elements as for classical triangles.

Let  $\{\mathcal T_h\}$ be a family of triangulations of $\O$ such that
any two triangles in ${\mathcal T}_h$ share at most a vertex or an edge and each element $T \in {\mathcal T}_h$. 
We denote by $ h_T$ and $\theta_T$ the greatest side and smallest angle of the triangle $T$.
We assume that the family of triangulations $\{ {\mathcal T}_h \}$ satisfies a minimum angle condition, i.e., there exists a constant $\theta_0 > 0$ such that $\theta_T \geq \theta_0$, for any $T \in {\mathcal T}_h$.  We also denote by 
$h = \max_{T\in {\mathcal T}_h} h_T$, i.e.,  the maximum diameter of the elements in $\matT_h$. 

We take $\bV_h \subset \bH^1(\O)$ and $Q_h \subset L^2_0(\O)$. We make further \emph{assumptions} for the finite element spaces $\bV_h\times Q_h$:\\
\begin{itemize}
\item[A1:] For each $\bw \in \bH^2(\O)\cap\{\phi\in\bH^1_0(\O):\div\phi=0\}$ and $q \in H^1(\O) \cap L^2_0(\O)$, there exist approximations $I_h(\bw) \in \bV_h$ and $\rho_h(q) \in Q_h$, such that
$$|\bw - I_h(\bw)|_{H^1(\O)} \leq Ch^r\|\bw\|_{H^2(\O)}, \quad \|q - \rho_h(q)\|_{L^2(\O)} \leq Ch^r\|q\|_{H^1(\O)},$$
for some $r >0$.

\item[A2:] And the so-called inf–sup inequality, i.e.,
\end{itemize}
\begin{align*}
\sup_{\bw_h\in\bV_h}\frac{\int_\O q_h\div\bw_h}{\|\bw_h\|_{H^1(\O)}}&\geq \beta\|q_h\|_{L^2(\O)} \quad \forall q_h\in Q_h\\
\sup_{q_h\in Q_h}\frac{\int_\O q_h\div\bw_h}{\|q_h\|_{L^2(\O)}}&\geq \hat{\beta}\|\bw_h\|_{H^1(\O)} \quad \forall \bw_h\in \bV_h
\end{align*}
 for some $\beta, \hat{\beta}>0$.

There are several families of elements present in the literature that satisfy conditions A1 and A2 (see, for example, \cite{Ar2018,AB2010,ABF1984,BDG2006,Boffi1995,BBDDFF2008}  and the references therein).
From now on, let us consider $\bg_\eps=\bg_h$, i.e. a regular approximation of $\bg$, such that $\bg_h$ is the trace of a function in $\bV_h$ and 

\begin{equation}\label{estimation_g-g_h}
\|\bg-\bg_h\|_{L^2(\Gamma)}\leq Ch^s ,    
\end{equation}
for some $s>0$.

The finite element
approximation of the Navier–Stokes equations \eqref{NV_dato_ge} is to find $(\bu_h,p_h)\in \bV_h\times Q_h$, with $\bu_h = \bg_h$ in $\Gamma$, such that
\begin{equation}\label{aprox-fem-ge}
\left\{\begin{array}{rll}
     a(\bu_h,\phi_h)+c(\bu_h,\bu_h,\phi_h)-b(\phi_h,p_h) &= \int_\O\bf\cdot\phi_h \quad &\forall \phi_h\in\bV_h\cap \bH^1_0(\O)  \\
     b(\bu_h,q_h) &= 0 \quad &\forall q_h\in Q_h
\end{array}
\right.
\end{equation}
  
The existence and uniqueness of problem \eqref{aprox-fem-ge} is studied, for example, in \cite{wang2016,gunzburger1983}, for conforming finite element methods that satisfy A1 and A2, under appropriate conditions on the data. 

In the next section, we consider some possible choices for the discrete spaces $\bV_h$, $Q_h$ and for the function $\bg_h$. 

The following result provides an useful a priori error estimate for the finite element approximation of Navier-Stokes problem with regular data satisfying the conditions of Theorem \ref{teo-existyunic-NS}.

\begin{prop}\label{estimate_result}
 Let $\bf\in\bH^{-1}(\Omega)$. If A1 and A2 hold, we get the following error estimates 
    \begin{equation*}\label{estimation_1}
        \|\bv-\bu_h\|_{L^2(\O)} \leq Ch^{2r}, 
    \end{equation*}
    \begin{equation}\label{estimation_2}
        |\bv-\bu_h|_{H^1(\O)}+\|q-p_h\|_{L^2(\O)}\leq Ch^r.
    \end{equation}
    with $(\bv,q)$ the solution of  \eqref{forma_debil_regular} and $(\bu_h,p_h)$ the solution of \eqref{aprox-fem-ge}.
\end{prop}
\begin{proof}
    See, for example, \cite[Theorem 3.2]{wang2016}.
\end{proof}

Using this last result together with Theorem \ref{main_result_2}, we can obtain the following main result.

\begin{theorem}\label{main_result_3}
 Let $\bf\in\bH^{-1}(\O)$ be and $\bg_h$ satisfying \eqref{cond-compatibilidad} and \eqref{estimation_g-g_h}.
    Let $\bu\in\bL_\Gamma$ be the solution of the problem \eqref{ec-principales}  and $\bu_h$ be the solution of \eqref{aprox-fem-ge}. If A1 and A2 hold, then, there exists a constant $C$ such that 
    \begin{equation*}
        \|\bu-\bu_h\|_{L^4(\O)}\leq C h^{\min\{r,s\}}
     \end{equation*}       
       
\end{theorem}

\begin{proof}
    Let $\bv\in\bH^1(\O)$ solution of the problem \eqref{forma_debil_regular}. For \eqref{inmersion} 
    $$\|\bv-\bu_h\|_{L^4(\O)}\leq C_I \|\bv-\bu_h\|_{H^1(\O)}$$
  Using Theorem \ref{main_result_2} and \eqref{estimation_2} we have
    \begin{align*}
        \|\bu-\bu_h\|_{L^4(\O)}&=\|\bu-\bv+\bv-\bu_h\|_{L^4(\O)}\\
        &\leq \|\bu-\bv\|_{L^4(\O)}+\|\bv-\bu_h\|_{L^4(\O)}\\
        &\leq C\frac{C_1\|\bg-\bg_h\|_{L^2(\Gamma)}}{1-C_1(1/ \nu)\|\bg-\bg_h\|_{L^2(\Gamma)}} + C_I\|\bv-\bu_h\|_{H^1(\O)}\\
        &\leq C\frac{C_1\|\bg-\bg_h\|_{L^2(\Gamma)}}{1-C_1(1/ \nu)\|\bg-\bg_h\|_{L^2(\Gamma)}}+O(h^r) 
    \end{align*}
Assuming that \eqref{estimation_g-g_h} holds, we get $\|\bu-\bu_h\|_{L^4(\O)}\leq C h^{\min\{r,s\}}$, for $h$ small enough, and the proof concludes.
\end{proof}

\section{Numerical experiments}\label{sec5: numerical experiments}

In this section we discuss some numerical aspects of our approach and present numerical experiments for the classic \textit{lid-driven cavity flow} problem, which is a typical benchmark problem in computational fluid dynamics \cite{duran2020,hamouda2017,AB2010}.

The first issue is to choose finite elements spaces $\bV_h$ and $Q_h$ satisfying conditions A1 and A2. 
We will consider the following:

\begin{itemize}

\item \textit{Taylor-Hood:} $\bV_h=(P_2(\matT_h))^2$ and $Q_h = P_1(\matT_h)\cap L^2_0(\O)$.
\item \textit{Mini-Elements:} $\bV_h=(P_1^b(\matT_h))^2$ and $Q_h = P_1(\matT_h)\cap L^2_0(\O)$ 

where, if for an element $T$, $b_T \in P_3$ is the classic cubic bubble function vanishing on $\partial T$, we set $$P_1^b(T)=P_1(T)\oplus span\{b_T(\cdot)\}$$

\item \textit{$\bP_1P_1$ stabilized:}  the method consists of find $(\bu_h,p_h)\in \bV_h\times Q_h$ with $\bu_h = \bg_h$ in $\Gamma$, $\bV_h = (P_1(\matT_h))^2$ and $Q_h=P_1(\matT_h)\cap L^2_0(\O)$ such that
\begin{equation*}\label{prob-aprox-fem}
\left\{\begin{array}{rll}
     a(\bu_h,\phi_h)+c(\bu_h,\bu_h,\phi_h)-b(\phi_h,p_h) &= \int_\O\bf\cdot\phi_h &\forall \phi_h\in\bV_h\cap \bH^1_0(\O)  \\
     b(\bu_h,q_h)-G(p_h,q_h) &= 0 & \forall q_h\in Q_h
\end{array}
\right.
\end{equation*}
where $$G(p,q)=\int_\O (p-\Pi p)(q-\Pi q)dx,$$
with $\Pi q|_T = \frac{1}{|T|}\int_T q dx, \forall T\in\matT_h$. This method was proposed in \cite{BDG2006} for the Stokes problem and studied for the Navier-Stokes problem in \cite{zheng,HL2008}.

\end{itemize}
After one of those methods is chosen, another question is how to select an appropriate function $\bg_h$ that satisfies \eqref{estimation_g-g_h} and the compatibility condition \eqref{cond-compatibilidad}.
One possibility is to take $\bg_h$ as the Cartensen interpolator of the function $\bg$ \cite{cartensen}. Another option is to take $\bg_h$ as the Lagrange interpolator but, unfortunately, this choice doesn't verify \eqref{cond-compatibilidad}. In \cite[lemma 3.1]{duran2020} the authors propose a modification of the Lagrange interpolator so that \eqref{cond-compatibilidad} holds.

Let $\O=[-1, 1]\times [-1, 1]$, $\bf=(0,0)$, $\nu=1$ and
\begin{equation*}
    \bg(x,y) = \left\{\begin{array}{cc}
         (0,0) & -1<x<1, \quad y=-1 \\
         (1,0) & y=1 
    \end{array}
    \right.
\end{equation*}

We will obtain the numerical approximations applying the three finite elements methods described above 
and choosing $\bg_h$ as the Lagrange interpolator.
We observe that, in this case, the compatibility condition is automatically verifed and furthermore, following \cite[Proposition 3.1]{duran2020}, we have that
\begin{equation*}\label{estimation_g-g_h_duran}
\|\bg-\bg_h\|_{L^2(\Gamma)}\leq Ch^{1/2}
\end{equation*}

In consequence,  applying Theorem \ref {main_result_3}, since  $r=1$ and $s=1/2$ for the proposed methods, we can affirm that

\begin{equation}\label{final-estimation}
    \|\bu-\bu_h\|_{L^4(\O)}\leq C h^{1/2}
\end{equation}

Other numerical aspect to consider is the 
iterative methods to approximate the discrete solution. The typical  algorithm to solve the Navier-Stokes problem is based on a fixed-point iterative algorithm. One possibility, that we will apply in our numerical experiments, is to use a Newton scheme, which has been studied for both the homogeneous \cite{girault1986,karakashian1982} and non-homogeneous cases \cite{wang2016,gunzburger1983}. 

Let $(\bu^n_h,p^n_h)$ $(n=0,1,2,...)$ be the sequence of approximations to $(\bu_h,p_h)$ defined by Newton's method defined as follows.

Given $(\bu^0_h,p^0_h)\in\bV_h\times Q_h$ such that $\bu^0_h|_\Gamma=\bg_h$, the sequence $(\bu^n_h,p^n_h)$ is defined by:

\begin{align}\label{newton-method}
    a(\bu^n_h,\phi_h)+c(\bu^n_h,\bu^{n-1}_h,\phi_h)&+c(\bu^{n-1}_h,\bu^n_h,\phi_h) -b(\phi_h,p^n_h) \nonumber \\
    &=\int_\O\bf\cdot\phi_h + c(\bu^{n-1}_h,\bu^{n-1}_h,\phi_h)
      \quad \forall \phi_h\in\bV_h\cap \bH^1_0(\O)\\
     b(\bu^n_h,q_h) &= 0 \quad \forall q_h\in Q_h \nonumber
\end{align}

In \cite[Proposition 3.4 and Proposition 3.5]{gunzburger1983} it is shown that if $\bu^0_h$ is close enough to $\bu_h$, then the sequence defined by \eqref{newton-method} is unique and converges quadratically to $\bu_h$. In other words, there exists a $\delta >0$ such that if $\|\bu_0-\bu_h\|_{H^1(\O)}\leq\delta$, then for $n\geq 1$, 
$$
    \|\bu^n_h-\bu_h\|_{H^1(\O)}\leq \frac{1}{\delta^{2n-1}}\|\bu^0_h-\bu_h\|^{2}_{H^1(\O)} 
    $$

Further, if $\|\bu^0_h-\bu_h\|_{H^1(\O)}=\delta\eta$, with $0<\eta<1$, then for $n\geq 1$
\begin{equation}\label{iterative-error2}
    \|\bu^n_h-\bu_h\|_{H^1(\O)}\leq \delta\eta^{2n}
\end{equation} 

\begin{corollary}
    Under the conditions of the Theorem \ref{main_result_3} and let  $(\bu^*_h,p^*_h)$ be the limit of the sequence given by \eqref{newton-method}. Then $$\|\bu-\bu^*_n\|_{L^4(\O)}\leq Ch^{1/2}.$$
\end{corollary}
\begin{proof}
From \eqref{iterative-error2}, we can assume that
there exists a constant $C$ such that $\|\bu_h-\bu^*_h\|_{L^4(\O)} \leq Ch$.
Then, since

$$\|\bu-\bu^*_h\|_{L^4(\O)}\leq \|\bu-\bu_h\|_{L^4(\O)}+\|\bu_h-\bu^*_h\|_{L^4(\O)}\leq \|\bu-\bu_h\|_{L^4(\O)}+C\|\bu_h-\bu^*_h\|_{H^1(\O)}.$$
the proof concludes as a direct consequence of \eqref{final-estimation}.

\end{proof}

The numerical implementation is based on the $NGSolve$ Python package \cite{NGSolve}.
For the application of Newton’s method, a stopping criterion must be prescribed in terms of a given tolerance. In the present work, we employ the relative error between two consecutive iterates: 

$$\frac{|\bu^n_h-\bu^{n-1}_h|_{H^1(\O)}}{|\bu^n_h|_{H^1(\O)}}<tol$$ with $tol=10^{-8}$. As an initial guess for the Newton iteration, we take the solution of the corresponding Stokes problem.

Next, for the different three methods under consideration, we estimate the convergence errors for $\bu$ in the $L^4(\O)$ norm. Since we do not know the exact solution and that our numerical solution is the result of an iterative method, the $L^4(\O)$-error is calculated as the difference between the limits solutions obtained in two consecutive refinements. 
In fact, let us note as $\bu_h^{\star}$ and  $\bu_{h/2}^{\star}$ the limits of the sequence $\bu^n_h$ for a mesh of size $h$ and $h/2$ respectively. Then, we define successive errors as,
$$e_{L^4}(\bu)= \| \bu^{\star}_{h} - \bu^{\star}_{h/2} \|_{L^4(\O)}$$

Figure \ref{convergence_order} and table \ref{tabla1} show the graphics of $\log(h)$ vs $\log(e_{L^4})$ for  $\mathbf{P}_2P_1$, $\mathbf{P}^b_1P_1$ and stabilized $\mathbf{P}_1P_1$ elements respectively, achieved
 with uniform mesh refinement, and also exhibit 
the experimental order of convergence (eoc) with respect to $h$, which is computed by  comparing the errors for two consecutive meshes.
We can observe that, in all the cases,  the numerical results confirm the theorical prediction.

\begin{figure}[H]
    \centering
    \includegraphics[width=0.65\linewidth]{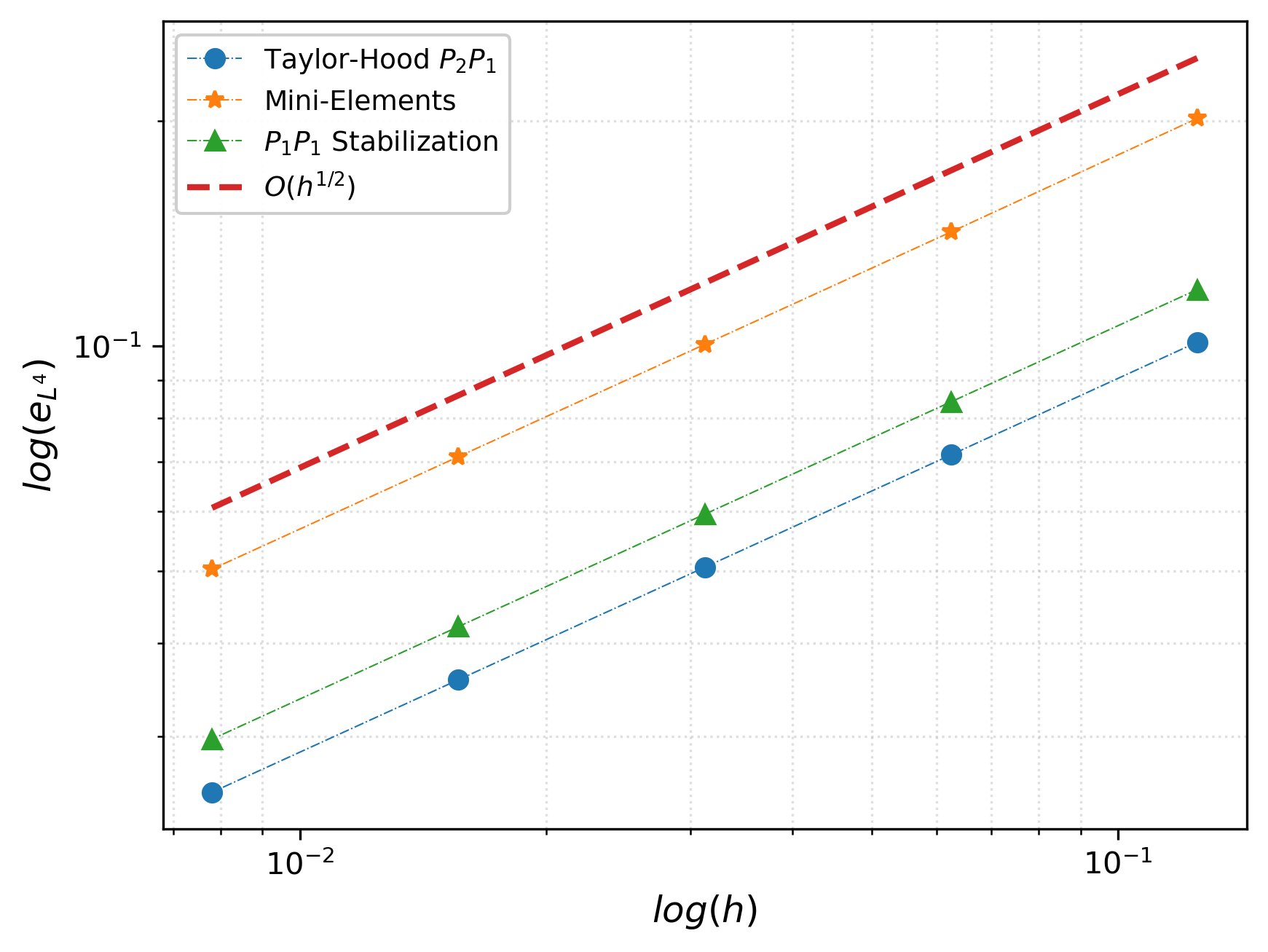}
    \caption{Convergence order analysis.}
    \label{convergence_order}
\end{figure}

\begin{table}[h]
\caption{$L^4$-norm errors and experimental orders of convergence for the approximation obtained with uniform mesh refinements, for different finite element discretizations.}\label{tabla1}
\begin{tabular*}{\textwidth}{@{\extracolsep\fill}lccccccccc}
\toprule%
& \multicolumn{2}{@{}c@{}}{$\mathbf{P}_2P_1$} & \multicolumn{2}{@{}c@{}}{$\mathbf{P}^b_1P_1$} & \multicolumn{2}{@{}c@{}}{$\mathbf{P}_1P_1$ stabilized}\\ \cmidrule(lr){2-3} \cmidrule(lr){4-5} \cmidrule(lr){6-7}
$h$ &  $e_{L^4}$ & eoc &  $e_{L^4}$ & eoc & $e_{L^4}$ & eoc  \\
\midrule
1/4  & 1.0116e-01  & -        & 2.0160e-01 & -       & 1.1894e-01 & -\\
1/8  & 7.1461e-02  & 0.50146  & 1.4205e-01 & 0.50514 & 8.4146e-02 & 0.49931\\
1/16 & 5.0534e-02  & 0.49991  & 1.0044e-01 & 0.50006 & 5.9510e-02 & 0.49978\\
1/32 & 3.5729e-02  & 0.50015  & 7.1057e-02 & 0.49924 & 4.2082e-02 & 0.49992\\
1/64  & 2.5262e-02 & 0.50010  & 5.0247e-02 & 0.49995 & 2.9757e-02 & 0.49998\\
\botrule
\end{tabular*}
\end{table}

\section{Conclusions}

The finite numerical approximation of the stationary Navier-Stokes equations with non-smooth data have been analyzed. We have compared the original problem with its corresponding version using smooth data, closely related to the original data, and the associated finite element approach.
Convergence and optimal a priori error estimates have been rigorously proven under suitable assumptions regarding the domain and finite element spaces. The reported numerical experiments for the \textit{lid-driven cavity flow} problem, with different finite element methods, have illustrated the effectiveness of our approach.

\bmhead{Acknowledgements}
We thank Ricardo G. Durán for his helpful references on a priori estimates for the Stokes problem, and Ariel L. Lombardi for his valuable suggestions regarding numerical implementation.

\section*{Declarations}

\noindent\textbf{Funding.} This work was partially supported by grant PICT 2018-3017, Agencia Nacional de Promoción de la Investigación, el Desarrollo Tecnológico y la Innovación (Argentina), DOI http://dx.doi.org/10.13039/501100021778 and by grant 20020170100056BA, Universidad de Buenos Aires, DOI \\ http://dx.doi.org/10.13039/501100005363.

\noindent\textbf{Competing interests.} The authors have no competing interests to declare that are relevant to the content of this article.\smallskip

\noindent\textbf{Data availability.}
Data sharing not applicable to this article as no datasets were generated or analysed during the current study.

\bibliography{references}
\end{document}